\documentclass[12pt,final]{amsart}
\usepackage{amssymb, amsmath, amsthm}
\usepackage{mathrsfs}  
\usepackage{graphicx}
\usepackage[ocgcolorlinks, linkcolor=blue]{hyperref}
\usepackage{geometry}
 \usepackage{xcolor}
 \usepackage{tcolorbox}
\usepackage{mathrsfs}  

\usepackage{bm}

\usepackage{ifdraft}
\ifoptionfinal{
\usepackage[disable]{todonotes}
}{
\usepackage[norefs, nocites]{refcheck}

\usepackage[notref, notcite]{showkeys}
\usepackage[bordercolor=white, color=white]{todonotes}
}

\makeatletter\providecommand\@dotsep{5}\def\listtodoname{List of Todos}\def\listoftodos{\hypersetup{linkcolor=black}\@starttoc{tdo}\listtodoname\hypersetup{linkcolor=blue}}\makeatother

\usepackage{hyperref}
\usepackage{etoolbox}

\newtheorem{theorem}{Theorem}[section]
\newtheorem{lemma}[theorem]{Lemma}

\theoremstyle{remark}
\newtheorem{remark}{Remark}

\numberwithin{equation}{section}

\def\C{\mathbb C}
\def\R{\mathbb R}

\def\N{\mathbb N}

\def\i{\textrm{i}}

\renewcommand{\leq}{\leqslant}
\renewcommand{\geq}{\geqslant}

\def\p{\partial}

\newcommand*\xbar[1]{%
   \hbox{%
     \vbox{%
       \hrule height 0.5pt 
       \kern0.5ex
       \hbox{%
         \ensuremath{#1}%
       }%
     }%
   }%
} 

\title[]{Reconstruction of 1-D evolution equations and their initial data from one passive measurement}

\author[A. Feizmohammadi]{Ali Feizmohammadi}
\address{Department of Mathematical and Computational Sciences,
	University of Toronto Mississauga, 3359 Mississauga Road
	Deerfield Hall, Mississauga, ON L5L 1C6.}
\email{ali.feizmohammadi@utoronto.ca}


\begin{document}


\maketitle

\begin{abstract}
We study formally determined inverse problems with passive measurements for one dimensional evolution equations where the goal is to simultaneously determine both the initial data as well as the variable coefficients in such an equation from the measurement of its solution at a fixed spatial point for a certain amount of time. This can be considered as a one-dimensional model of widely open inverse problems in photo-acoustic and thermo-acoustic tomography. We provide global uniqueness results for wave and heat equations stated on bounded or unbounded spatial intervals. Contrary to all previous related results on the subject, we do not impose any genericity assumptions on the coefficients or initial data. Our proofs are based on creating suitable links to the well understood spectral theory for 1D Schr\"odinger operators. In particular, in the more challenging case of a bounded spatial domain, our proof for the inverse problem partly relies on the following two ingredients, namely (i) a Paley-Wiener type theorem for Schr\"odinger operators due to Remling  \cite{Remling2002SchrdingerOA} and a theorem of Levinson \cite{Levinson1940} on distribution of zeros of entire functions of regular growth that together provide a quantifiable link between support of a compactly supported function and the upper density of its vanishing Schr\"odinger spectral modes and (ii) a result of Gesztesy and Simon \cite{Gesztesy1999InverseSA} on partial data inverse spectral problems for reconstructing an unknown potential in a 1D Schr\"odinger operator from the knowledge of only a fraction of its spectrum.  
\end{abstract} 


\section{Introduction}

\subsection{Motivation} The theory of inverse problems for PDEs seeks to recover hidden parameters of a system from indirect measurements. In many settings, the observer can actively probe the system and measure its response—this leads to inverse problems with active measurements. As an example, propagation of acoustic waves through an isotropic medium is governed by the equation 
$$\Box_c u:= \p^2_t u - c^2(x)\Delta u =0 \quad \text{on $\R_+\times \R^n$},$$
where $\R_+=(0,\infty)$. An experimental setup would be to probe an unknown medium $\Omega$ by sending artificial wave packets and subsequently measuring the reflected waves that arrive back to the surface of the medium.  This would be equivalent to the knowledge of the following dataset 
$$ \mathcal C(c) = \left\{(u,\p_\nu u)|_{\R_+\times \p\Omega}\,:\, u\in C^\infty(\R_+\times \R^n),\quad \Box_c u=0 \quad \text{on $\R_+\times \R^n$}\right\},$$
where $\nu$ is the outward unit normal vector field on $\p \Omega$. The inverse problem would be to uniquely determine the wave speed $c(x)$ inside the unknown medium, namely $x\in \Omega$, from the dataset $\mathcal C(c)$. This problem was solved by Belishev in \cite{Bel87} through the introduction of the powerful boundary control method, see e.g. \cite{Belishev2007RecentPI,Belishev2008BoundaryCA,Nur23} for further results and expositions on the boundary control method. In fact, it was proved in \cite{Katchalov2002EquivalenceOT} that analogous inverse problems for various type of evolution equations are all equivalent to each other.

In many real-world applications, a more physically realistic model is that of passive measurements, where the observer has no control over the excitation and must infer internal properties solely from the system’s spontaneous response. In these models, the observer can only measure the response of the system to an unknown cause and is nevertheless tasked with the question of identifying the properties of the medium, as well as the initial causal factor.  An example of this type of inverse problem would be to obtain information about the interior structure of an unknown object, such as earth, from passive measurements of waves on its boundary, such as the seismic data from an active earthquake happening at an unknown location inside the earth. As another well known example in photo-acoustic and thermo-acoustic tomography, one is allowed to probe a medium with an electromagnetic wave which in turn produces heat and an elastic expansion thus generating a sound wave that the observer can subsequently measure at the boundary of the medium. In photo-acoustic tomography (PAT) this is done with a rapidly pulsating laser
beam; in thermoacoustic tomography (TAT) the medium is probed with an electromagnetic
wave of a lower frequency; see e.g. \cite{Kruger1999ThermoacousticCT}. In many practical situations the sound speed inside the medium is unknown
and in order to determine this sound speed additional measurements have been considered in
\cite{Jin2006ThermoacousticTW}. To describe the latter inverse problem in more detail, consider the acoustic wave equation,
\begin{equation}\label{pat}
	\begin{aligned}
		\begin{cases}
			\p^2_t u-c^2(x)\,\Delta u=0\,\quad &\text{$(t,x)\in \R_+\times \R^n$}
			\\
			u(0,x)=f(x) \,\quad &\text{$x\in \R^n$}\\
			\p_t u(0,x)=0 \,\quad &\text{$x\in \R^n$}
		\end{cases}
	\end{aligned}
\end{equation}
where $f$ models the initial pressure (generally assumed to be not identically zero) and $c$ is the speed of sound \cite{Diebold1991PhotoacousticMR}. In practice both of these are a priori unknown and the task is to determine them both uniquely from the passive boundary measurements $u(t,x)$ with $t\in \R_+$ and $x\in \p\Omega$. 

Due to the very limited amount of data, the above inverse problem is notoriously difficult and most of the known results impose rather stringent assumptions. For example, certain results assume that the wave speed is known and it is only the initial pressure that is sought after, see e.g. \cite{Agranovsky2007OnRF,Belishev2019OnAI,Oksanen2013PhotoacousticAT,FaouziTriki2025,Stefanov2009ThermoacousticTW} for related results. On the other hand, there are some results where the authors use one passive measurement from a specially designed (and often highly singular) a priori known initial data $f$ to determine the unknown coefficient in the PDE, see e.g. \cite{BukKli81,Cheng2002IdentificationOC,Cheng2009UniquenessIA,Feizmohammadi2020GlobalRO,Helin2013InversePF,Helin2016CorrelationBP,Kian2020TheUO,Kian2022SimultaneousDO,Rakesh2001AOI,Stefanov2011RecoveryOA}. For the general problem where both the initial data and the coefficient of the PDE are unknown, the available results either assume that the wave speed and initial data both satisfy some explicit equations \cite{Knox_2020,Liu_2015}, that they obey certain monotonicity conditions \cite{Kian23}, or finally that they are a priori known to belong to an abstract generic class, see e.g. \cite{Avdonin2011ReconstructingTP,Finch2013TransmissionEA,Jing2022,Pierce1979UniqueIO,Murayama1981TheGT,Suzuki1983UniquenessAN,Suzuki1986InversePF}.

In this paper, we are concerned with one dimensional models of inverse problems with passive measurements such as the photo-acoustic and thermo-acoustic tomography inverse problems discussed above. In this 1D setup, and by introducing a new perspective based on deep known results in spectral theory, we provide comprehensive results for the simultaneous recovery of coefficients and the initial data without resorting to any genericity type assumptions that have been present in all previously related 1D inverse problems with passive measurements such as \cite{Avdonin2011ReconstructingTP,Jin2006ThermoacousticTW,Jing2022,Suzuki1986InversePF,Pierce1979UniqueIO}. We hope the analysis will shed some light on the related multi-dimensional inverse problems. 

\subsection{Inverse problem for waves in an unbounded domain}

We begin with a one-dimensional model of inverse problems in photo-acoustic tomography. Let $c\in C^2([0,\infty))$ be a positive function satisfying
\begin{equation}
	\label{c_cond}
	\textrm{supp}\, (c_0-c) \subset (0,1),
\end{equation}
for some positive constant $c_0>0$. We consider the one dimensional wave equation on half-line,
\begin{equation}\label{pf0}
	\begin{aligned}
		\begin{cases}
			\p^2_t u-c^2(x)\,\p^2_x u=0\,\quad &\text{$(t,x)\in \R_+^2:=(0,\infty)^2$}
			\\
			u(0,x)=f(x) \,\quad &\text{$x\in \R_+$}\\
			\p_t u(0,x)=0 \,\quad &\text{$x\in \R_+$}\\
			u(t,0)=0  \,\quad &\text{$ t\in \R_+$.}
		\end{cases}
	\end{aligned}
\end{equation}
For the initial data, we make the assumption that $f\in H^2(\R_+)$ and that
\begin{equation}\label{F_supp}
	\textrm{supp}\,f\, \subset (0,1).
	\end{equation} 
Given any $f$ as above, it is well known that equation \eqref{pf0} admits a unique solution $u$ in the energy space
\begin{equation}
	\label{natural_space}
	C^2(\R_+;L^2(\R_+))\cap C^1(\R_+;H^1(\R_+))\cap C^0(\R_+;H^2(\R_+)).
\end{equation}
Moreover, $u(\cdot,x) \in H^2(\R_+)$ and $\p_xu(\cdot,x)\in H^1(\R_+)$ for each fixed $x\in \R^+$, see e.g. \cite[Proposition 1.1]{Arnold2022OnTE}. We consider the passive measurement
\begin{equation}
	\label{passive}
	(f,c)\mapsto u(t,1) \quad t\in \R_+,
\end{equation}
where $u$ is the unique solution to \eqref{pf0} subject to the initial data $f$ and the wave speed $c$. We are interested in the following inverse problem; is it possible to simultaneously and uniquely determine the wave speed $c$, as well as initial data $f$, given the passive boundary measurement \eqref{passive} of the unique solution to \eqref{pf0}? We prove the following theorem that settles the inverse problem above under the natural assumption that the initial data satisfies \eqref{F_supp} and that it is not identically zero.
\begin{theorem}
	\label{thm_1}
	For $j=1,2,$ let $f_j\in H^2(\R_+)$ be not identical to zero, let $c_j\in C^{2}(\R_+)$ be a positive function and assume that \eqref{c_cond}--\eqref{F_supp} are satisfied with $c=c_j$, $f=f_j$ and $u=u_j$. Then,
	$$u_1(t,1)=  u_2(t,1) \quad \forall\,t\in \R_+ \implies  (f_1,c_1)\equiv(f_2,c_2). $$
\end{theorem}

This result goes beyond prior work on inverse problems with passive measurements in 1D, where recovery of both the coefficient and initial data typically requires genericity or monotonicity assumptions.



\subsection{Inverse problem for heat flows in a bounded domain}

Let $T>0$. We consider a one dimensional convection-diffusion process modelled by the parabolic initial boundary value problem 
\begin{equation}\label{pf2}
	\begin{aligned}
		\begin{cases}
			\p_t u-\p^2_xu+b(x)\p_x u=0\,\quad &\text{on $(0,T)\times (0,1)$},\\
			u(t,0)=0 \,\quad &\text{on $(0,T)$,}\\
			u(t,1)=0 \,\quad &\text{on $(0,T)$,}\\
				u(0,x)=g(x) \,\quad &\text{on $(0,1)$,}
		\end{cases}
	\end{aligned}
\end{equation}
Given each $g\in H^1_0((0,1))$, it is classical that the above initial boundary value problem admits a unique solution 
$$u\in H^1(0,T;L^2((0,1)))\cap L^2(0,T;H^2((0,1))),$$ 
see e.g. \cite[Theorem 4.3]{LionsMagenes1972}. Moreover, $\p_x u(t,1) \in H^{\frac{1}{4}}((0,T))$. For our second inverse problem, the goal is to simultaneously determine the convection function $b$ as well as the initial data $g$ from the single passive measurement
$$ (g,b) \mapsto  \p_x u(t,1) \quad t\in [T_1,T_2],$$
for some $0\leq T_1<T_2<T$.

\begin{theorem}
	\label{thm_2}
	Let $T>0$ and let $\{t_k\}_{k=1}^\infty\subset \R$ satisfy $\lim\limits_{k\to \infty}t_k=T_1$ for some $T_1\in (0,T)$. For $j=1,2,$ let $b_j\in C^{1}([0,1])$ and let $g_j \in H^1_0((0,1))$ be not identical to zero. Assume that 
	\begin{equation}\label{supp_cond}
		\textrm{supp}\, g_1 \cup 	\textrm{supp}\, g_2 \subset [0,\varepsilon] \quad \text{and} \quad \textrm{supp}\,(b_1-b_2) \subset [0,\frac{1}{2}-\frac{\varepsilon}{2}),		
		\end{equation}
	for some $\varepsilon\in (0,1)$. Denote by $u_j$ the unique solution to \eqref{pf2} with $g=g_j$ and $b=b_j$. Then,
	$$\p_xu_1(t_k,1)=  \p_xu_2(t_k,1) \quad \forall\,k\in \N \implies  (g_1,b_1)\equiv(g_2,b_2).$$
\end{theorem}

\begin{remark}
	\label{rmk_result}
	Our methodology for the proof of the above theorem is rather robust and can allow certain generalizations. For example, instead of the convection term $b$, one may consider a conductivity type equation 
	$$ \p_t u - \p_x \left( \sigma(x) \p_x u\right)=0 \quad (t,x)\in (0,T)\times (0,1).$$
	The result above would still follow after some modifications on the assumption \eqref{supp_cond} that would involve some a priori bounds on the conductivity term $\sigma$. As another example, one may consider different Dirichlet, Neumann or mixed spatial boundary conditions. In fact, it would even be possible to add an unknown Robin type boundary condition such as
	$$ \sin(\alpha)u(t,0) +\cos(\alpha) \p_x u(t,0)=0 \quad t\in (0,T),$$
	for some unknown $\alpha \in (-\frac{\pi}{2},\frac{\pi}{2}]$ instead of the Dirichlet condition above and subsequently aim to recover $(\alpha,b,g)$ uniquely. For simplicity of presentaion, we have avoided discussing these further results here. 
\end{remark}

\section{Proof of Theorem~\ref{thm_1}}

	We will assume throughout this section that the assumptions of Theorem~\ref{thm_1} are fulfilled.  
	
	\begin{remark}\label{rmk_u}
	It is possible to extend the solutions $u_1(t,x)$ and $u_2(t,x)$ to the full spacetime $t\in \R$ and $x\in \R$ simply by considering an even extension of the wave speeds $c_j(x)$, $j=1,2,$ to the full space $x\in \R$ and an odd extension of $u_j(t,x)$ to $\R_+\times \R$ followed by an even extension in time to the full spacetime $\R\times \R$. For the remainder of this section, we use the same notation $u_j$ for this extended solution. Furthermore, by combining known local  decay estimates in time for wave equations, see e.g. \cite{Arnold2022OnTE}, together with Sobolev embedding we deduce that 
			\begin{equation}
			\label{exp_decay}
			|u_j(t,x)| + 	|\p_x u_j(t,x)| \leq C\, e^{-\kappa |t|} \quad x\in [-1,1] \quad t\in \R \quad j=1,2,
		\end{equation} 
	for some $C,\kappa>0$ depending on $\|f_1\|_{H^2(\R)}$, $\|f_2\|_{H^2(\R)}$, $\|c_1\|_{C^2(\R)}$ and $\|c_2\|_{C^2(\R)}$.
\end{remark}

\begin{lemma}
	\label{lem_cauchy}
	We have $u_1(t,1)=u_2(t,1)$ and $\p_x u_1(t,1) = \p_x u_2(t,1)$ for all $t\in \R$.
\end{lemma}

\begin{proof}
	The first equality is a trivial consequence of Remark~\ref{rmk_u}. Next, note that the function $w(t,x)= u_1(t,x)-u_2(t,x)$ satisfies
	\begin{equation*}
		\begin{aligned}
			\begin{cases}
				\p^2_t w-c_0^2\,\p^2_x w=0\,\quad &\text{$(t,x)\in \R\times (1,\infty)$}
				\\
				w(0,x)=0 \,\quad &\text{$x\in (1,\infty)$}\\
				\p_t w(0,x)=0 \,\quad &\text{$x\in (1,\infty)$}\\
				w(t,1)=0  \,\quad &\text{$ t\in \R$.}
			\end{cases}
		\end{aligned}
	\end{equation*}
By uniqueness of solutions to the above PDE, we conclude that $w(t,x)=0$ for all $(t,x)\in \R\times (1,\infty)$.
\end{proof}

\begin{lemma}
	\label{lem_energy}
	Let $c\in C^2([0,1])$ be a positive function. Given any $z \in \C$, define $v(z,\cdot)$ to be the unique solution to the following ODE
	\begin{equation}\label{eq_system} -c^2(x)\,\p^2_x v(z,x) = z^2 v(z,x) \quad \text{on $\R$}\quad \text{and $v(z,0)=0$, $\p_x v(z,0)=1$.} \end{equation}
	There exists positive constants $C>0$ depending only on $\|c\|_{C^2([0,1])}$ such that
		\begin{equation}
		\label{ODE_energy}
		C^{-1} e^{-\lambda |z|} \leq |v(z,1)| + |\p_x v(z,1)| \leq C e^{\lambda |z|}, \quad \forall\, z \in \C.
	\end{equation}
\end{lemma}

\begin{proof}
	We will consider equation \eqref{eq_system} on the interval $[0,1]$. Let us introduce the functions $p\in C^2([0,1])$ and $w(z,\cdot)\in C^2([0,1])$ for each $z\in \C$ via 
	\begin{equation}\label{p_w_def}
	p(x) = \int_0^x \frac{1}{c(s)}\,ds \quad \text{and}\quad w(z,x)= e^{-\mathrm{i}p(x)z}\,v(z,x) \quad \forall\, x\in [0,1].
	\end{equation}
	For the rest of this proof we will fix $z\in \C$. We also hide the dependence of $w(z,\cdot)$ on $z$ simply writing $w$ in its place. Rewriting \eqref{eq_system} in terms of the above functions, we see that
	\begin{equation}\label{w_ODE}
		 \p^2_x w + 2\mathrm{i}zp'(x)\p_x w + \mathrm{i}zp''(x) w=0 \quad \forall\, x\in [0,1],
		\end{equation}
	subject to $w(z,0)=0$ and $\p_x w(z,0)=1$. Let us now define for each fixed $z\in \C$,
	\begin{equation}\label{def_E}
		E_z(x) = \frac{1}{2} \left( |w(z,x)|^2+ |\p_x w(z,x)|^2 \right) \quad \forall\, x\in [0,1]. 
		\end{equation}
	Note that
	\begin{equation}\label{E_zero}
			E_z(0)=\frac{1}{2}\quad \forall\, z\in \C.
		\end{equation}
	Differentiating $E_z(x)$ with respect to $x$ and using \eqref{w_ODE} we deduce that
	$$
	E_z'(x)= \mathrm{Re}\, \left( w \overline{w'} + \overline{w'} w'' \right)= \mathrm{Re}\,\left( w\overline{w'} -2\mathrm{i}zp' w' \overline{w'} -\mathrm{i}zp'' \overline{w'}{w}  \right).
	$$
	It follows that there exists a constant $C_0>0$ independent of $z\in \C$ such that
	$$
	|E'_z(x)| \leq C_0\left(1+|z|\right) E_z(x) \quad \forall\, x\in [0,1] \quad \forall\, z\in \C.
	$$
	Applying Gronwall's inequality together with \eqref{E_zero}, it follows that there exists constants $C_1,\lambda_1>0$ independent of $z\in \C$, such that
	\begin{equation}\label{E_est_1}
	C_1^{-1}\, e^{-\lambda_1|z| x}	\leq  E_z(x) \leq C_1 e^{\lambda_1 |z| x} \quad \forall\, x\in [0,1] \quad \forall\, z\in \C.
		\end{equation} 
	Using Cauchy-Schwarz inequality, it follows from the latter inequality together with definition of $E_z(x)$ that there exists a constant $C_2,\lambda_2>0$ such that
		\begin{equation}\label{E_est_2}
		C_2^{-1}\, e^{-\lambda_2|z| x}	\leq |w(z,x)| + |\p_x w(z,x)| \leq C_2 e^{\lambda_2 |z| x} \quad \forall\, x\in [0,1] \quad \forall\, z\in \C.
	\end{equation} 
Next, recalling the definition \eqref{p_w_def} we write for each $z\in \C$ and each $x\in [0,1]$,
\begin{equation}\label{v_w_rel}
 |v(z,x)| + |\p_x v(z,x)|= \left| e^{\mathrm{i}\,zp(x)}\right| \cdot \left( |\p_x w + \mathrm{i}zp'(x)w| + |w|\right).
\end{equation}
Note that there exists $C_3>0$ independent of $z\in \C$ such that
$$
C_3^{-1} (1+|z|)^{-1} (|w|+ |\p_x w|)	\leq  |\p_x w + \mathrm{i}zp(x)w| + |w| \leq C_3 (1+|z|) (|w|+ |\p_x w|),
$$
for all $x\in [0,1]$ and all $z\in \C$. The desired bound \eqref{ODE_energy} follows immediately from combining the latter bound with \eqref{E_est_2} and \eqref{v_w_rel}  and setting $x=1$.
\end{proof}

\begin{lemma}
	\label{lem_reduction_wave}
	Assume that the hypotheses of Theorem~\ref{thm_1} is satisfied. For each $j=1,2$ and any $z \in \C$, define $v_{j}(z,\cdot)$ as in \eqref{eq_system} with $c=c_j$. There exists a nowhere vanishing function $\rho:\R\to \R$ such that
	$$ v_{2}(\xi,x)=\rho(\xi)\,v_{1}(\xi,x) \quad x\geq 1 \quad  \xi \in \R.$$
\end{lemma}

\begin{proof}
	We begin by recording that in view of Lemma~\ref{lem_energy} (with $c=c_j$, $j=1,2,$), there holds
	\begin{equation}
		\label{ODE_1}
		C^{-1} e^{-\lambda |z|} \leq |v_{j}(z,1)| + |\p_x v_{j}(z,1)| \leq C e^{\lambda |z|}, \quad z \in \C,
	\end{equation}
	for $j=1,2,$ and some $C>0$ and $\lambda>0$ independent of $z$.  
	
	Next, recalling Remark~\ref{rmk_u}, let us define for each $j=1,2$ and each 
	$$(z,x)\in \C\times [0,1]\quad \text{with} \quad |\mathrm{Im}\,z|<\kappa,$$ 
	the function
	$$
	\hat{u}_j(z,x) = \int_{\R} u_j(t,x)\,e^{-\i z t} \,dt.
	$$
	Owing to \eqref{exp_decay}, we note that given any fixed $x\in [0,1]$, the function $\hat{u}_j(z,x)$ is analytic in the strip $|\textrm{Im}(z)|<\kappa$. Moreover, for any $x\in [0,1]$ and $z\in \C$ with $|\textrm{Im}(z)|<\kappa$, we have 
	\begin{equation}\label{u_j_1}
		-c_j^2(x)\p^2_x \hat{u}_j(z,x)= z^2\,\hat{u}_j(z,x),
	\end{equation}
	and that
	\begin{equation}
		\label{u_j_2}
		\hat{u}_j(z,0)=0 \quad \text{and} \quad \p_x \hat{u}_j(z,0) = a_j(z):=\int_\R \p_x u_j(t,0)\,e^{-\i z t}\,dt.
	\end{equation}
	Note that the function $a_j(z)=\int_\R \p_x u(t,0)\,e^{-\i z t}\,dt$ is well defined and analytic in the strip $|\textrm{Im}(z)|<\kappa$, thanks to \eqref{exp_decay}.  We claim that for each $j=1,2,$ the set 
	\begin{equation}\label{J_j}
	J_j=\{z\in \C\,:\, |\textrm{Im}(z)|\leq \frac{\kappa}{2}, \quad \text{and}\quad a_j(z)=0\},		
	\end{equation}
	is a collection of isolated points. Indeed, if this was not the case,  then there would be an accumulation point within the strip $|\textrm{Im}(z)|\leq \frac{\kappa}{2}$ and thus by analyticity, it would follow that $a_j(\xi)=0$ for all $\xi \in \R$ and subsequently by taking inverse Fourier transform that $\p_x u_j(t,0)=0$ for all $t\in \R$. Together with the fact that  $u_j(t,0)=0$ for all $t\in \R$, it would imply through energy estimates for one dimensional wave equations that $u_j(t,x)\equiv 0$ which is a contradiction to the fact that $f_j$ is not identical to zero. 

Let us now return to \eqref{u_j_1}-\eqref{u_j_2} and note as a consequence that given any $x\in [0,1]$ there holds
	\begin{equation}
		\label{v_hat_u}
		\hat{u}_j(z,x) = a_j(z) \, v_{j}(z,x), \qquad |\textrm{Im}(z)|\leq \frac{\kappa}{2}.
	\end{equation}
	Recalling Lemma~\ref{lem_cauchy}, we deduce that
	$$
	\hat{u}_1(z,1)=\hat{u}_2(z,1) \quad  \text{and}\quad \p_x\hat{u}_1(z,1)=\p_x\hat{u}_2(z,1), \quad |\textrm{Im}(z)|\leq \frac{\kappa}{2},
	$$
	which can be rewritten as
	\begin{equation}\label{key_iden_u}
		a_1(z)\,\left(v_{1}(z,1),\p_x v_{1}(z,1)\right)=	a_2(z)\,\left(v_{2}(z,1),\p_x v_{2}(z,1)\right),
		\end{equation}
	for all  $z\in \C$ with $|\textrm{Im}(z)|\leq \frac{\kappa}{2}$. Let us define the set 
	$$\mathbb A= \{z\in \C\,:\, |\textrm{Im}(z)| \leq \frac{\kappa}{2}\}\setminus (J_1\cup J_2).$$ 
As a consequence of \eqref{key_iden_u}, we obtain that 
$$
|a_1(z)|\cdot \left(|v_{1}(z,1)|+|\p_x v_{1}(z,1)|\right)=	|a_2(z)|\cdot \left(|v_{2}(z,1)|+|\p_x v_{2}(z,1)|\right).
$$
Recall from \eqref{ODE_1} that both terms in the parenthesis are nonzero, which allows us to obtain the following equation
$$
	\left|\frac{a_1(z)}{a_2(z)}\right| = \frac{|v_2(z,1)|+|\p_x v_2(z,1)|}{|v_1(z,1)|+|\p_x v_1(z,1)|} \qquad z\in \mathbb A.
$$
Using \eqref{ODE_1} again, we deduce that
\begin{equation}\label{key_iden_2}	\frac{C_0}{C_1}\,e^{-2\lambda |z|}\leq \left|\frac{a_1(z)}{a_2(z)}\right| \leq \frac{C_1}{C_0}\, e^{2\lambda |z|},\end{equation}
for all $z\in \mathbb A$. As the function $(\frac{a_1(z)}{a_2(z)})$ is holomorphic in a punctured neighborhood of any $z\in J_1\cup J_2$, and as it satisfies the bounds given in \eqref{key_iden_2}, it follows from Riemann's theorem for removable singularities that the function $(\frac{a_1(z)}{a_2(z)})$ admits a holomorphic extension, denoted by $\rho(z)$ here, to the entire strip $|\textrm{Im}(z)|\leq \frac{\kappa}{2}$, and that this extended holomorphic function is nowhere vanishing in this strip. This implies that 
$$ \left(v_{2}(z,1),\p_x v_{2}(z,1)\right)=	\rho(z)\,\left(v_{1}(z,1),\p_x v_{1}(z,1)\right), \quad |\textrm{Im}(z)|\leq \frac{\kappa}{2}.$$
The claim now follows from restricting the above equality to the positive real-axis in the complex plane.
\end{proof}

Before proving the theorem, we first require a lemma concerning the asymptotic expression of the Dirichlet eigenvalues for the operator \( -c^{-2}(x) \frac{d^2}{dx^2} \) on the interval \( (0,1) \). While this result is likely well known, we were only able to locate such expansions in the context of one-dimensional Schr\"odinger operators. Thus, for the reader’s convenience, we include a proof here. Let us also recall that the operator \( -c^2(x)\frac{d^2}{dx^2} \) with $c\in C^2([0,1])$ has a discrete Dirichlet spectrum, see e.g. \cite[Chapter 1]{Zettl2021} for an introduction into the topic. 

\begin{lemma}
	\label{lem_sturm}
	Let $c\in C^2([0,1])$ be a positive function. Let us denote by $\mu_1<\mu_2<\ldots$ the Dirichlet eigenvalues for the differential operator $-c^{2}(x)\p^2_x$ on the interval $(0,1)$. There holds,
	\begin{equation}\label{eigen_asymp_c}
		\mu_k = \left(\int_0^1 \frac{1}{c(s)}\,ds\right)^{-2}\, k^2 \pi^2 + O(1), \quad \text{as $k\to \infty$}.
		\end{equation}
	\end{lemma}

\begin{proof}
 To verify the spectral asymptotics \eqref{eigen_asymp_c}, we will first perform a standard Liouville type transformation that will change the ODE \eqref{eq_system} to a Schr\"odinger equation. We recall the Liouville transformation law as follows,
 \begin{equation}\label{liouville_eq}
-\p_x \left( \sigma\,\p_x (\sigma^{-\frac{1}{2}}u)\right) = \sigma^{\frac{1}{2}}\, \left(-\p^2_x u + \frac{\p^2_x{\sigma^{\frac{1}{2}}}}{\sigma^{\frac{1}{2}}} u \right),
 \end{equation}
for all positive functions $\sigma \in C^2([0,1])$ and all $u\in C^2([0,1])$. Motivated by this transformation, let us define a new variable $t$ by 
	$$ t(x) = \frac{1}{L}\int_0^{x}\frac{1}{c(s)}\,ds,$$
	where 
	$$ L:= t(1)=\int_0^{1} \frac{1}{c(s)}\,ds>0,$$
	and we have normalized the variable $t$ so  that $t(0)=0$, $t(1)=1$. Note also that $t$ is a strictly increasing function of $x$ for all $x\in \R_+$. Defining
	$$ \tilde{v}(\xi,t)= L^{-1}\,c(0)^{-\frac{1}{2}}\,c^{-\frac{1}{2}}(x(t))\,v(\xi,x(t)) \quad \text{and} \quad \tilde{c}(t):=c(x(t)),$$
	we can rewrite the ODE given by \eqref{eq_system} (thanks to \eqref{liouville_eq}) as follows
	$$ -\p^2_{t} \tilde{v}(\xi,t) + q(t) \tilde{v}(\xi,t) = L^2\,\xi^2\,  \tilde{v}(\xi,t) \quad \text{on $\R_+$},$$
	subject to $\tilde{v}(\xi,0)=0$ and $\p_{t}\tilde{v}(\xi,0)=1$, where 
	\begin{equation}\label{q_def_pf}
		q = \tilde{c}^{\frac{1}{2}}\,\p^2_{t} \tilde{c}^{-\frac{1}{2}}.\end{equation}
	Let us denote by $\lambda_1<\lambda_2<\ldots$, the Dirichlet eigenvalues for the operator $-\p^2_t + q(t)$ on the interval $(0,1)$. It is clear from the above reduction that there holds:
	$$\mu_n = L^{-2}\,\lambda_n \quad \forall\, n\in \N.$$
	The theorem now follows from the well known asymptotic expressions for Dirichlet eigenvalues of 1-D Schr\"odinger operators on $(0,1)$ with a bounded potential, see for example  \cite[Theorem 4, page 35]{Pschel1986InverseST}.
\end{proof}

We are ready to prove the theorem.

\begin{proof}[Proof of Theorem~\ref{thm_1}]
	Recalling Lemma~\ref{lem_reduction_wave}, we deduce that 
	\begin{equation}\label{eigen_equality} \left (v_1(\xi,1)=0 \iff v_2(\xi,1)=0 \right) \quad \forall\, \xi \in \R.\end{equation}
	For $j=1,2,$ denote by $\{\mu_k^{(j)}\}_{k=1}^\infty$ the Dirichlet eigenvalues of $-c_j^2(x)\p^2_x$ on the interval $(0,1)$ written in strictly increasing order and observe from \eqref{eigen_equality} that there holds:
	$$  \mu_k^{(1)}= \mu_k^{(2)} \quad k\in \N.$$
	Applying Lemma~\ref{lem_sturm} we deduce that
	\begin{equation}\label{weyl_asymp}
				L:=\int_0^1 \frac{1}{c_1(x)}\,dx = \int_0^1 \frac{1}{c_2(x)}\,dx. 
		\end{equation}
	Let us now return to the equation 
	$$ -c_j^2(x)\,\p^2_xv_j(\xi,x) = \xi^2 v_j(\xi,x) \quad x\in \R_+ \quad \xi \in \R,$$ 
	subject to $v_j(\xi,0)=0$ and $\p_xv_j(\xi,0)=1$. We will perform a similar Liouville transformation as in Lemma~\ref{lem_sturm} that will change the ODE above to a Schr\"odinger equation. Fixing $j\in \{1,2\}$, let us define a new variable $y_j$ by 
	$$ y_j(x) = \int_0^{x}\frac{1}{c_j(s)}\,ds.$$
	Recalling that $c_0=c_1(0)=c_2(0)$, we define
	$$ \tilde{v}_j(\xi,y_j) := c_0^{-\frac{1}{2}}\,c_j^{-\frac{1}{2}}(x(y_j))\,v_j(\xi,x(y_j)) \quad \text{and} \quad \tilde{c}_j(y_j):=c_j(x(y_j)),$$
	we can rewrite the above ODE as follows
	$$ -\p^2_{y_j} \tilde{v}_j(\xi,y_j) + q_j(y_j) \tilde{v}_j(\xi,y_j) = \xi^2 \tilde{v}_j(\xi,y_j) \quad \text{on $\R_+$},$$
	subject to $\tilde{v}_j(\xi,0)=0$ and $\p_{y_j}\tilde{v}_j(\xi,0)=1$, where 
	$$ q_j = \tilde{c}_j^{\frac{1}{2}}\,\p^2_{y_j} \tilde{c}_j^{-\frac{1}{2}}.$$
	With the reduction to Schr\"odinger equation now achieved, we note crucially that 
	\begin{equation}\label{y_eq}  y_1(1) = 	\int_0^1 \frac{1}{c_1(x)}\,dx = \int_0^1 \frac{1}{c_2(x)}\,dx=y_2(1).\end{equation}
	Recalling Lemma~\ref{lem_reduction_wave} together with the above equality, we note that
	$$  \tilde{v}_2(\xi,y) =\rho(\xi) \,\tilde{v}_1(\xi,y) \quad \xi \in \R, \quad  y\geq L:=y_1(1)=y_2(1).$$
	As $\rho$ is nowhere vanishing, we note that
		\begin{equation}\label{eigen_equality_2} \left (\tilde{v}_1(\xi,2L)=0 \iff \tilde{v}_2(\xi,2L)=0 \right) \quad  \xi \in \R.\end{equation}
		The above equality implies that the 1D Schr\"odinger operators 
		$$ -\p^2_{y} +q_1(y)  \quad \text{and}\quad  -\p^2_{y} +q_2(y), $$
	have the same Dirichlet eigenvalues on the interval $[0,2L]$. As $\textrm{supp}(q_1-q_2) \subset (0,L)$, it follows from \cite{Hochstadt1978ANIS} that $q_1(y)=q_2(y)$ for all $y\in [0,L]$. Therefore,
	$$   \tilde{c}_1^{\frac{1}{2}}\,\p^2_{y} \tilde{c}_1^{-\frac{1}{2}}= \tilde{c}_2^{\frac{1}{2}}\,\p^2_{y} \tilde{c}_2^{-\frac{1}{2}} \quad y\in [0,L].$$
	This reduces to 
	$$
	-\frac{1}{2}\p_y^2 \log(\tilde{c}_1) +\frac{1}{4} (\p_y \log(\tilde{c}_1))^2 =-\frac{1}{2}\p_y^2 \log(\tilde{c}_2) +\frac{1}{4} (\p_y \log(\tilde{c}_2))^2,    \quad    y\in [0,L],
	$$
	and subsequently to 
	$$
  \p^2_y \log(\frac{\tilde{c}_1}{\tilde{c}_2}) -A(y)\p_y\log(\frac{\tilde{c}_1}{\tilde{c}_2})=0   \quad   y\in [0,L],
	$$
	where $A(y)=\frac{1}{2}\p_y\log(\tilde{c}_1(y)\,\tilde{c}_2(y))$. As $\tilde{c}_1(0)=\tilde{c}_2(0)=c_0$ and $\tilde{c}_1(L)=\tilde{c}_2(L)=c_0$ (thanks to \eqref{y_eq}), we deduce via unique existence of solutions to ODEs that 
	$$ \tilde{c}_1(y)=\tilde{c}_2(y) \quad  y \in [0,L].$$
	Next, for $j=1,2,$ defining $x_j:[0,L]\to [0,1]$ via 
	$$ x_1(y) := \int_0^{y}\tilde{c}_1(s)\,ds=\int_0^{y}\tilde{c}_2(s)\,ds=:x_2(y),$$
	and recalling from the definition of $\tilde{c}_j$ that $c_j(x_j(y))=\tilde{c}_j(y)$ for all $y\in [0,L]$, we deduce from above that
	$$c_1(x)=c_2(x) \quad  x\in [0,1].$$
	Finally, the equality of the initial data follows from unique continuation principle for wave equations.
\end{proof}


\section{Proof of Theorem~\ref{thm_2}}
We will assume throughout this section that the assumptions of Theorem~\ref{thm_2} are fulfilled. For $j=1,2$ and $x\in [0,1]$, let us define 
\begin{equation}\label{vV}
 v_j(x)= e^{-\frac{1}{2}\int_1^xb_j(s)\,ds} u_j(x) \quad \text{and}\quad V_j(x)= -\frac{1}{2}b_j'(x)+\frac{1}{4}(b_j(x))^2.\end{equation}
We emphasize here that $V_j$ is real-valued and that $V_j\in C^0([0,1])$. It is straightforward to see that 
\begin{equation}\label{heat_3}
	\begin{aligned}
		\begin{cases}
			\p_t v_j-\p^2_xv_j+V_j(x)\,v_j=0\,\quad &\text{on $(0,T)\times (0,1)$},\\
			v_j(t,0)=0 \,\quad &\text{on $(0,T)$,}\\
			v_j(t,1)=0 \,\quad &\text{on $(0,T)$,}\\
			v_j(0,x)=h_j(x):=e^{-\frac{1}{2}\int_1^xb_j(s)\,ds} g_j(x) \,\quad &\text{on $(0,1)$,}
		\end{cases}
	\end{aligned}
\end{equation}
We note that in view of the statement of Theorem~\ref{thm_2} there holds
\begin{equation}\label{new_eq_v}
		\p_xv_1(t_k,1)=\p_x v_2(t_k,1) \quad  k\in \N.
	\end{equation}
Let us denote by $\lambda^{(j)}_1<\lambda^{(j)}_2<\ldots$ the Dirichlet eigenvalues associated to $-\p^2_x +V_j$ and by $\{\phi^{(j)}_k\}_{k=1}^{\infty}\subset C^{2}([0,1])$ an orthonormal $L^2((0,1))$-basis consisting of eigenfunctions of $-\p^2_x + V_j$ on $(0,1)$, that is to say, 
$$ -\p^2_x\phi^{(j)}_k(x) +V_j(x) \phi^{(j)}_k(x) = \lambda^{(j)}_k\, \phi^{(j)}_k(x) \quad x\in I=(0,1),$$
subject to 
$$ \phi_k^{(j)}(0)=\phi_k^{(j)}(1)=0 \quad \text{and}\quad \|\phi_k^{(j)}\|_{L^2((0,1))}=1\quad k\in \N.$$
\begin{lemma}
	\label{lem_spectral_heat}
	Given any $k\in \N$, if $\int_{0}^1 h_1(x)\,\phi_k^{(1)}(x)\,dx \neq 0$, then 
	$$ \lambda_k^{(1)} \in \{\lambda_l^{(2)}\}_{l=1}^{\infty}.$$ 
\end{lemma}
\begin{proof}
Using the above spectral decomposition, we may write for each $j=1,2$ and $t\in (0,T)$,
\begin{equation}
	\label{u_exp}
	\p_x v_j(t,1)= \sum_{k=1}^{\infty} \left(\int_0^1 h_j(x)\phi_k^{(j)}(x)\,dx\right)\, e^{-\lambda^{(j)}_kt}\, \p_x\phi^{(j)}_k(1),
\end{equation}
Let us define for each $j=1,2$ and $t>0$, the function
\begin{equation}\label{w_j_exp}
w_j(t)= \sum_{k=1}^{\infty} a_k^{(j)} e^{-\lambda^{(j)}_kt} \quad t\in \R_+,
\end{equation}
where
\begin{equation}
	\label{def_a_j}
	a_k^{(j)}=\left(\int_0^1 h_j(x)\phi_k^{(j)}(x)\,dx\right)\, \p_x\phi^{(j)}_k(1) \quad \forall\, k\in \N.
\end{equation}
It is straightforward to see that $w_j$ depends real-analytically on $t\in \R_+$. Recalling \eqref{new_eq_v} we obtain that
\begin{equation}\label{new_eq_v_boundary}
w_1(t)=w_2(t) \quad t\in \R_+.
\end{equation}
Let us study the behaviour of the coefficients $a_k^{(j)}$ in a bit more detail. First, we recall from \cite[Theorem 4, page 35]{Pschel1986InverseST} the asymptotic pointwise eigenfunction estimates for 1D Schr\"odinger operators given by 
\begin{equation}\label{eigen_pw}
	\begin{aligned}
		\phi_k^{(j)}(x) &= \sqrt{2}\sin(k\pi x) + O(\frac{1}{k}),\\
		\p_x \phi_k^{(j)}(x)&= \sqrt{2}k\pi \cos(k\pi x) + O(1),\\
	\end{aligned}
\end{equation}
as $k\to \infty$ where the convergence is uniform in $x\in [0,1]$, as well as the asymptotic eigenvalue expression
\begin{equation}\label{eigen_asymp}
	\lambda^{(j)}_k = (k\pi)^2 + \int_0^1 V_j(x)\,dx + o(1) \quad \text{as $k  \to \infty$}.
\end{equation}
By combining the above asymptotic expressions with integration by parts and using the fact that $h_j(0)=h_j(1)=0$, we may write for each $k>k_0$, with $k_0>0$ sufficiently large, that
\begin{multline}\label{h_bound}
	\left| \int_0^1 h_j(x)\phi_k^{(j)}(x)\,dx\right|= \frac{1}{\lambda_k^{(j)}}\left| \int_0^1 h_j(x)(-\p^2_x+V_j)\phi_k^{(j)}(x)\,dx\right|\\
	= \frac{1}{\lambda_k^{(j)}}\left| \int_0^1 \p_xh_j(x)\,\p_x\phi_k^{(j)}+ \int_0^1 h_j(x)V_j(x)\,\phi_k^{(j)}(x)\,dx)\right|\\
	\leq2\frac{1+\|V_j\|_{L^\infty((0,1))}}{\lambda_k^{(j)}}\|h_j\|_{H^1((0,1))}\, \|\phi_k^{(j)}\|_{H^1((0,1))} \leq \frac{C}{k},
\end{multline}
for some constant $C>0$ that is independent of $k$. Note that we are using the bounds \eqref{eigen_pw} and \eqref{eigen_asymp} in the last step above. Recalling the definition of $a_k^{(j)}$ together with \eqref{h_bound} as well as \eqref{eigen_pw} for $\p_x \phi_k^{(j)}(1)$, we obtain that
\begin{equation}\label{a_k_bound} |a_k^{(j)}| \leq C_0 \qquad \forall\, k\in \N\end{equation}
for some positive constant $C_0$ independent of $k$. Let us now return to the definition of $w_j(t)$, $t>0$ with $j=1,2$. In view of the above uniform bound for $a_k^{(j)}$ together with the eigenvalue bounds $\frac{1}{\lambda_k} =O(\frac{1}{k^2})$ as $k\to \infty$, we see that given each 
$$z\in \C \quad \text{with} \quad \textrm{Re}(z)>\max\{-\lambda_1^{(1)},-\lambda_2^{(1)}\},$$  
the function 
\begin{equation}\label{w_abs} w_j(t)\, e^{-zt}= \sum_{k=1}^\infty a_k^{(j)} e^{(-\lambda^{(j)}_k-z)t},\end{equation}
converges absolutely in $L^1(\R_+)$. Next, we define the Laplace transform
$$ (\mathcal Lw_j)(z) = \int_0^\infty w_j(t)\, e^{-zt}\,dt \quad \forall\, z\in \C\quad \text{with $\textrm{Re}(z)>\max\{-\lambda_1^{(1)},-\lambda_2^{(1)}\}$},$$
Using the expression \eqref{w_abs}, we obtain via the Dominated Convergence Theorem that
$$
	(\mathcal Lw_j)(z)=\sum_{k=1}^{\infty} \frac{a_k^{(j)}}{z+ \lambda_k^{(j)}} \quad \forall\, z\in \C \quad \text{with} \quad \textrm{Re}(z)>\max\{-\lambda_1^{(1)},-\lambda_2^{(1)}\}.$$  
Recalling the equality \eqref{new_eq_v_boundary} it follows that
\begin{equation}
	\label{spectral_exp}
	\sum_{k=1}^{\infty} \frac{a_k^{(1)}}{z+ \lambda_k^{(1)}}=	\sum_{k=1}^{\infty} \frac{a_k^{(2)}}{z+ \lambda_k^{(2)}} \quad  \textrm{Re}(z)>\max\{-\lambda_1^{(1)},-\lambda_2^{(1)}\}.
\end{equation}
Using the uniform bound \eqref{a_k_bound} again together with the eigenvalue asymptotics \eqref{eigen_asymp}, we deduce that the function
$$
H_j(z)= \sum_{k=1}^{\infty} \frac{a_k^{(j)}}{z+ \lambda_k^{(j)}},
$$
is holomorphic on $\C$ away from the collection of isolated points $\{-\lambda_k^{(j)}\}_{k=1}^{\infty}$. Combining \eqref{spectral_exp} with analytic continuation, we deduce that
\begin{equation}\label{H_eq}
H_1(z) = H_2(z) \quad z\in \C \setminus \{-\lambda_1^{(1)},-\lambda_{1}^{(2)},\ldots\}.
\end{equation}
The above equality has a simple but crucial consequence. If, $a_k^{(1)}\neq 0$ for some $k\in \N$, then $H_1(z)$ has a simple pole at $z= -\lambda^{(1)}_k$. By the equality \eqref{H_eq}, it follows that $H_2$ must also have a pole at $z=-\lambda^{(1)}_k$, which in turn implies that $\lambda_k^{(1)}$ must also be a Dirichlet eigenvalue for $-\p^2_x + V_2$ on $(0,1)$.
\end{proof}

Let us now recall a deep result of Levinson \cite[Chapter II, Theorem VIII, page 13]{Levinson1940} regarding the distribution of zeros of entire functions under certain growth estimates (see also \cite[Chapter 4, page 173]{Levin64}). Before writing the statement of the theorem of Levinson, we define the function $\log^+:\R \to [0,\infty)$ via
$$
\log^+x = \begin{cases}
	\log x \quad \text{if $x>1$}\\
	0		\quad \text{otherwise}.
\end{cases}
$$
\begin{theorem}[Chapter II, Theorem VIII, \cite{Levinson1940}]
	\label{thm_inter}
	Let $F(z)$ be an entire function such that it is not identical to zero. Assume that
	\begin{equation}\label{logarith_growth}
		\int_\R 	\frac{\log^+|F(x)|}{1+x^2}\,dx<\infty,
		\end{equation}
	and that
		\begin{equation}\label{lim_sup_cond}
		\limsup_{r\to \infty}\frac{\log |F(re^{\mathrm{i}\theta})|}{r} \leq k.
	\end{equation}
Let $n(r)$ be the number of zeros of the function $F(z)$ that lie in the region 
$$\{z\in \C\,:\, \mathrm{Re}\,z\geq 0\quad |z|< r \}.$$ 
Then, there exists a number $B\leq \frac{k}{\pi}$ such that, 
$$
\lim_{r\to \infty} \frac{n(r)}{r}=B \leq \frac{k}{\pi}.
$$
\end{theorem}
For our purposes, the existence of the limit above is not needed and we only are concerned with the consequence that $\limsup_{r\to\infty}  \frac{n(r)}{r}\leq \frac{k}{\pi}.$ We have the following immediate corollary of the above Theorem that will be used later in the proof of Theorem~\ref{thm_2}.
\begin{lemma}
	\label{lem_inter}
	Let $I=[-a,a]$ for some $a>0$ and let $f\in L^2(I)$ be not identical to zero. Denote by $\Lambda= \{\mu_k\}_{k=1}^\infty\subset (0,\infty)$ the set of those zeros of the entire function
	$$ F(z):=\int_I f(x)\,e^{-\mathrm{i}zx}\,dx,$$
	that lie on the positive real axis of the complex plane. There holds,
	$$ \limsup_{r\to \infty}\frac{\left|\Lambda \cap [0,r)\right|}{r} \leq \frac{a}{\pi}.$$
\end{lemma}
\begin{proof}
	Throughout this proof, we write $\log(z)$ for the logarithm function with its principal branch, namely $\arg(z) \in (-\pi,\pi]$. It is classical that the Fourier transform of a $L^2(\R)$ function with compact support is an entire function. This shows that the function $F$ defined in the lemma is indeed an entire function. Let us now verify the hypotheses of Theorem~\ref{thm_inter} for the entire function $F(z)$ defined in the lemma.  First, note that \eqref{logarith_growth} follows from the fact that
	$$
	\int_\R |F(x)|^2\,dx = \int_I |f(x)|^2\,dx,
	$$
	together with Cauchy-Schwarz inequality. Indeed,
	$$
	\left(\int_\R \frac{\log^+|F(x)|}{1+x^2}\,dx\right)^2 \leq 	\left(\int_\R \frac{|F(x)|}{1+x^2}\,dx\right)^2 \leq  \left(\int_\R |F(x)|^2dx\right)\,\left(\int_\R \frac{1}{(1+x^2)^2}\,dx\right)  <\infty.
	$$
	We will now show that \eqref{lim_sup_cond} is satisfied with $k=a$. To this end, first we note that for $\theta=0,\pm \pi$ in \eqref{lim_sup_cond}, there holds
	\begin{equation}\label{theta_0}
		\limsup_{r\to \infty}\frac{\log |F(re^{\mathrm{i}\theta})|}{r}=0 \quad \theta=0,\pm \pi.
	\end{equation}
	Next, for each $\theta \in (-\pi,\pi)\setminus \{0\}$ and each $r>0$, we write
	$$
	|F(re^{\mathrm{i}\theta})|=\left| \int_I f(x)\,e^{-\mathrm{i}\,re^{\mathrm{i}\theta}x}\,dx  \right| \leq \|f\|_{L^2(I)} \,\left( \int_{-a}^a e^{2rx\sin\theta}\,dx \right)^{\frac{1}{2}}\leq  \|f\|_{L^2(I)}\, \frac{e^{ra|\sin\theta|}}{{\sqrt{r}\sqrt{|\sin\theta|}}}.
	$$
	Using the fact that $\log z = \log|z| + \mathrm{i} \arg(z)$, it follows from the latter estimate that
	\begin{equation}\label{theta_not_zero}\limsup_{r\to \infty}\frac{\log |F(re^{\mathrm{i}\theta})|}{r} \leq a|\sin\theta|\leq a \quad \forall\,\theta \in (-\pi,\pi)\setminus \{0\}.\end{equation}
	We have verified that the hypotheses of Theorem~\ref{thm_inter} is satisfied with $k=a$. The result follows immediately.
\end{proof}

We have the following key lemma which crucially relies on the above mentioned theorem and another Payley-Wiener type result due to Remling \cite{Remling2002SchrdingerOA,Remling2003InverseST}. The idea here is that, due to the fact that the initial data is assumed to be supported in a small interval $[0,\epsilon]$, and thanks to the latter two results in spectral theory, we can show that only information on a certain {\em fraction} (in the sense of upper density) of the spectrum will be lost.

\begin{lemma}
	\label{lem_density_heat}
	For each $\lambda \in \R$, let us define
	$$ d(\lambda)= \left|  \{k\in \N\,:\, 0<\lambda_k^{(1)}<\lambda \quad \text{and}\quad \int_0^1 h_1(x)\,\phi_k^{(1)}(x)\,dx=0\} \right|.$$
	There holds,
	$$
	\limsup_{\lambda\to \infty} \frac{d(\lambda)}{\sqrt\lambda} \leq \frac{\varepsilon}{\pi}.
	$$
\end{lemma}
\begin{proof}
Consider the 1D Schr\"{o}dinger equations
$$ -\p^2_x\psi(z,x) + V_1(x) \psi(z,x) = z^2 \psi(z,x) \quad \text{on $(0,1)$},$$
subject to 
$$ \psi(z,0)=0 \quad \text{and} \quad \p_x \psi(z,0)=1.$$
Let us define for each $N\in (0,1)$,
$$ \mathcal S_N:= \left\{ \int_0^N f(x)\,\psi(z,x)\,dx\,:\, f\in L^2((0,N))\right\}.$$
Here, the parameter $N\in (0,\infty)$ is fixed and the set $\mathcal S_N$ is concerned with behaviour of truncated integrals of $f$ against eigenfunctions of the Schr\"odinger operator on $(0,1)$ from $x=0$ to $x=N$. In view of \cite[Theorem 2.2]{Remling2003InverseST} (cf. \cite[Theorem 4.1]{Remling2002SchrdingerOA}), for each fixed $N>0$, the set $\mathcal S_N$ is independent of $V_1$, that is to say,
$$ \mathcal S_N=\left\{\int_{0}^N f(x)\,\frac{\sin(z x)}{z}\,dx\,:\, f\in L^2((0,N))\right\}.$$
Thus, given a not identically vanishing $h_1$ as above (recall that $\textrm{supp}\, h_1 \subset [0,\varepsilon]$), there exists a not identically vanishing function $\tilde{f}\in L^2((0,\varepsilon))$ such that 
$$
\int_{0}^{\varepsilon} \tilde{f}(x)\,\frac{\sin(z x)}{z}\,dx = \int_0^{\varepsilon} h_1(x)\,\psi(z,x)\,dx, \quad \forall\, z\in \C.
$$
In particular, this means that for each $\lambda_k^{(1)}>0$ (we recall that $\lambda_1^{(1)}<\lambda_2^{(1)}<\ldots$ denotes the Dirichlet eigenvalues of $-\p_x^2 +V_1(x)$ on $(0,1)$), we have that
$$ \int_0^1 h_1(x) \, \phi_k^{(1)}(x)\,dx=0$$
if and only if 
$$  \int_0^\varepsilon \tilde{f}(x) \,\sin(\sqrt{\lambda_k^{(1)}} x)\,dx=0.$$ 
Let us define a new function $f\in L^{2}(\R)$ as follows
$$
f(x)= \begin{cases}
	\tilde{f}(x) \quad &\text{if $x\in (0,\varepsilon)$,}\\
	-\tilde{f}(-x) \quad &\text{if $x\in (-\varepsilon,0)$}\\
	0 \quad &\text{otherwise}.
\end{cases}
$$
In terms of the not identically vanishing function $f\in L^2(\R)$ with $\textrm{supp}\,f\subset [-\varepsilon,\varepsilon]$, we note that for any $\lambda_k^{(1)}> 0$, we have the equivalence 
\begin{equation}\label{h_condition} \int_0^1 h_1(x) \, \phi_k^{(1)}(x)\,dx=0\end{equation}
if and only if 
\begin{equation}\label{f_condition} (\mathcal Ff)(\sqrt{\lambda_k^{(1)}})=\int_{-\varepsilon}^{\varepsilon} \,f(x) e^{-\i \sqrt{\lambda^{(1)}_k} x}\,dx=0.\end{equation}
We can apply Lemma~\ref{lem_inter} (with $I=[-\varepsilon,\varepsilon]$) to the equation \eqref{f_condition} to obtain that 
$$
\limsup_{\lambda\to \infty} \frac{d(\lambda)}{\sqrt{\lambda}} \leq \frac{\varepsilon}{\pi}.
$$
\end{proof}
We are ready to prove the theorem. The idea is that by the above two lemmas, a certain fraction of spectrum of $-\p^2_x +V_1$ and $-\p^2_x+V_2$ are equal to each other and the aim is to reconstruct the potentials from such a limited knowledge. This is a variant of the classical inverse spectral problem for Schr\"odinger operators with a very rich literature, see e.g. \cite{Ambarzumian1929berEF,Borg1946EineUD,Hald1984DiscontinuousIE,Hochstadt1978ANIS,Levitan1964DeterminationOA,Simon1999ANA} for the case of the knowledge of the full spectral data and \cite{Hatinouglu2019MixedDI,Gesztesy1999InverseSA,Horvath2001OnTI,Horvath2005InverseSP,Marletta_2005} for results with only partial knowledge of the spectral data.

\begin{proof}[Proof of Theorem~\ref{thm_2}]
	By the assumption of the theorem, we have that 
	$$b_1(x)=b_2(x) \qquad x\in [\frac{1}{2}-\frac{\varepsilon_1}{2},1]$$ 
	for some $\varepsilon_1\in (\varepsilon,1)$. Recalling \eqref{vV}, we obtain 
	\begin{equation}
		\label{V_12}
		V_1(x)=V_2(x) \quad  x\in [\frac{1}{2}-\frac{\varepsilon_1}{2},1].
	\end{equation}
	Let us now define for each $\lambda\in \R$, the functions
	$$ N(\lambda)= \left|  \{k\in \N\,:\, 0<\lambda_k\leq \lambda\} \right|,$$
	and 
	$$ S(\lambda)= \left|  \{k\in \N\,:\, 0<\lambda_k^{(1)}\leq \lambda \quad \text{and}\quad \lambda_k^{(1)}=\lambda_l^{(2)} \quad \text{for some $l\in \N$}\} \right|.$$
It is clear from \eqref{eigen_asymp} that there holds
\begin{equation}\label{eigen_distribution}
	N(\lambda) = \frac{\sqrt{\lambda}}{\pi} + O(1).
\end{equation}
Let $\varepsilon_2 \in (\varepsilon,\varepsilon_1)$. By applying Lemma~\ref{lem_spectral_heat} and Lemma~\ref{lem_density_heat} together with \eqref{eigen_distribution}, we deduce that there exists $\lambda_0>0$ such that
$$ S(\lambda)= N(\lambda)-d(\lambda) \geq (1-\varepsilon_2) N(\lambda) \geq (1-\varepsilon_1)N(\lambda) + \frac{\varepsilon_1}{2} \qquad  \lambda>\lambda_0$$
where we recall that $\varepsilon_1 \in (\varepsilon,1)$. Thanks to \eqref{V_12} and the above inequality, the assumptions of \cite[Theorem A.3]{Gesztesy1999InverseSA} are satisfied and therefore, it follows from that theorem of Gesztesy and Simon that $V_1=V_2$ on $(0,1)$. Recalling \eqref{vV}, it follows that the function $b(x)=b_1(x)-b_2(x)$ satisfies
$$ b'(x) - (\frac{b_1(x)+b_2(x)}{2})\, b(x)=0 \quad  x\in (0,1).$$
As $b(1)=0$, it follows that $b(x)=0$ for all $x\in (0,1)$. The equality of the initial data now follows from unique continuation principle for heat equations.
\end{proof}

\subsection*{Conflict of interest}
The author has no conflicts of interest to declare that are relevant to this article.

\subsection*{Data availability statement}
Data sharing not applicable to this article as no datasets were generated or analysed during the current study.

	\bibliography{refs} 
	
	\bibliographystyle{alpha}

\end{document}